\documentclass[12pt]{amsart}
\usepackage{amssymb, verbatim}
\usepackage{epsfig,color}
\usepackage{enumerate}
\usepackage{color,url}

\date{\today}

\newtheorem{theorem}{Theorem}[section]

\theoremstyle{definition}

\theoremstyle{remark}

\makeatletter
\@namedef{subjclassname@2020}{%
 \textup{2020} Mathematics Subject Classification}
\makeatother

\numberwithin{equation}{section}

\usepackage[T1]{fontenc}

\newcommand{\hyper}[5]{\,{}_{#1}F_{#2}\left(\!\!%
\begin{array}{cc}{\displaystyle{#3}}\\[-0.1ex]
{\displaystyle{#4}} \end{array}\Big| \,{\displaystyle{#5}}
\right)}

\begin{document}

\title[Fourier transform of the orthogonal polynomials on the unit ball]{Fourier transform of the orthogonal polynomials on the unit ball}

\author[G\"{u}ldo\u{g}an Lekes\.{I}z]{Esra G\"{u}ldo\u{g}an Lekes\.{I}z}
\email[G\"{u}ldo\u{g}an Lekesiz]{esragldgn@gmail.com}

\author[Akta\c{s}]{Rab\.{I}a  Akta\c{s}}
\address[Akta\c{s}]{Ankara University, Faculty of Science, Department of Mathematics, 06100, Tando\u{g}an, Ankara, Turkey}
\email[Akta\c{s}]{raktas@science.ankara.edu.tr}

\author[Area]{Iv\'an Area}
\address[Area]{Universidade de Vigo,
		Departamento de Matem\'atica Aplicada II,
              E. E. Aero\-n\'au\-ti\-ca e do Espazo,
              Campus As Lagoas s/n,
              32004 Ourense, Spain}
\email[I. Area]{area@uvigo.gal}


\begin{abstract}
Fourier transform of multivariate orthogonal polynomials on the unit ball are obtained. By using Parseval's identity,  a new family of multivariate orthogonal functions are introduced. The results are expressed in terms of the continuous Hahn polynomials.
\end{abstract}

\maketitle

\section{Introduction}

{}From a historical point of view, mathematical transforms started with some works of L. Euler within the context of second-order differential equation problems \cite{Deakin}. Since then, due to their interesting mathematical properties as well as their applications, integral transforms have attracted research interests  in many areas of engineering, mathematics, physics, as well as several other scientific branches. Just to give an idea, without completeness, integral transforms such as Fourier, Laplace, Beta, Hankel, Mellin and Whittaker transforms with various special functions as kernels play an important role in various problems of physics, engineering, mathematics \cite{1,2,3,4,5,6,7,8,9,10,11,12}, and in vibration analysis, sound engineering, communication, data analysis, automatization, etc. \cite{13,14,15,16,17,18}.

As for the relation between orthogonal polynomials and integral transforms, by the Fourier transform or other integral transforms, it is shown that some systems of univariate orthogonal polynomials  are mapped into other family \cite{4}. For example, Hermite functions which are  Hermite polynomials $H_{n}\left(  x\right)  $\ multiplied by $\exp \left(  -x^{2}/2\right)  $\ are eigenfunctions of Fourier transform \cite{7,8,9,19}. Some other interesting works are related with families of classical discrete orthogonal polynomials \cite{Natig}. In \cite{9}, by the Fourier-Jacobi transform, it is investigated that classical Jacobi polynomials can be mapped onto Wilson polynomials. Also, Fourier transform of Jacobi polynomials and their close relation with continuous Hahn polynomials have been discussed by Koelink \cite{7}.

Recently, in the univariate case, the Fourier transforms of finite classical orthogonal polynomials by Koepf and Masjed-Jamei \cite{8}, generalized ultraspherical and generalized Hermite polynomials and symmetric sequences of finite orthogonal polynomials  \cite{11,20,21} have been studied. As for the multivariate case, Tratnik \cite{Trat1, Trat2} presented multivariable generalization both of all continuous and discrete families of the Askey tableau, providing hypergeometric representation, orthogonality weight function which applies with respect to subspaces of lower degree and biorthogonality within a given subspace. A non-trivial interaction for multivariable continuous Hahn polynomials has been presented by Koelink et al. \cite{Koe}. Moreover, in \cite{22,23,24} Fourier transforms of multivariate orthogonal polynomials and their applications have been investigated, obtaining some families of orthogonal functions in terms of continuous Hahn polynomials. In particular, in \cite{22} a new family of orthogonal functions has been derived by using Fourier transforms of bivariate orthogonal polynomials on the unit disc and Parseval's identity.


The main aims of this investigation are to find the Fourier transformation of the classical orthogonal polynomials on the unit ball $\mathbb{B}^{r}$ and to obtain a new family of multivariate orthogonal functions in terms of multivariable Hahn polynomials.

The work is organized as follows. In section \ref{sec:bdn} basic definitions and notations are introduced. Main results are stated and proved in section \ref{sec:ft}.

\section{Preliminaries}\label{sec:bdn}
In this section, we state background materials on orthogonal polynomials that we shall need. The first subsection recalls the properties of two families of (univariate) orthogonal polynomials, namely the Gegenbauer polynomials and the continuous Hahn polynomials, as well as some definitions. In the second subsection, we recall the basic results on the (multivariate) classical orthogonal polynomials on the unit ball.

\subsection{The classical univariate Gegenbauer Polynomials}
Let
\begin{equation}\label{1}
P_{n}^{\left(  \alpha,\beta \right)  }\left(  x\right)  =2^{-n}\sum\limits_{k=0}^{n}\binom{n+\alpha}{k}\binom{n+\beta}{n-k}\left(  x+1\right)^{k}\left(  x-1\right)  ^{n-k}
\end{equation}
be the univariate Jacobi polynomial of degree $n$, orthogonal with respect to the weight function \cite[p. 68, Eq. (4.3.2)]{244}
\[
w\left(  x\right)  =\left(  1-x\right)  ^{\alpha}\left(  1+x\right)  ^{\beta},~~\alpha,\beta>-1,~~x\in \left[  -1,1\right]  .
\]
The univariate Gegenbauer polynomials are a special case of Jacobi polynomial, defined by \cite[p. 277, Eq. (4)]{25}
\begin{equation}\label{3}
C_{n}^{\left(  \lambda \right)  }\left(  x\right)  =\frac{\left(2\lambda \right)  _{n}}{\left(  \lambda+\frac{1}{2}\right)  _{n}}P_{n}^{\left(\lambda-\frac{1}{2},\lambda-\frac{1}{2}\right)  }\left(  x\right),
\end{equation}
where for $n \geq 1$, $\left(  \alpha \right)  _{n}=\alpha \left(  \alpha+1\right) \cdots \left(
\alpha+n-1\right)$, denotes the
Pochhammer symbol with the convention $(\alpha)_{0}=1$.  These polynomials can also be written in terms of hypergeometric series as
\begin{equation}\label{hyper}
C_{n}^{\left(  \lambda \right)  }\left(  x\right)  =\frac{\left(
2\lambda \right)  _{n}}{n!}\ _{2}F_{1}\left(
\genfrac{}{}{0pt}{0}{-n,n+2\lambda}{\lambda+\frac{1}{2}}
\mid \frac{1-x}{2}\right) ,
\end{equation}
where \cite[p. 73, Eq. (2)]{25}
\begin{equation}\label{genhyper}
_{p}F_{q}\left(
\genfrac{}{}{0pt}{0}{a_{1},\ a_{2},\dots,a_{p}}{b_{1},\ b_{2},\dots,b_{q}}\mid x\right)  =
{\displaystyle \sum \limits_{n=0}^{\infty}}
\frac{\left(  a_{1}\right)  _{n}\left(  a_{2}\right)  _{n}\dots \left(
a_{p}\right)  _{n}}{\left(  b_{1}\right)  _{n}\left(  b_{2}\right)
_{n}\dots \left(  b_{q}\right)  _{n}}\frac{x^{n}}{n!}.
\end{equation}
The Gegenbauer polynomials satisfy the orthogonality relation \cite[p. 281, Eq. (28)]{25}
\begin{equation}\label{ort}
 \int \limits_{-1}^{1}\left(  1-x^{2}\right)  ^{\lambda-\frac{1}{2}}
C_{n}^{\left(  \lambda \right)  }\left(  x\right)  C_{m}^{\left(
\lambda \right)  }\left(  x\right)  dx=h_{n}^{\lambda}~\delta_{n,m}
, \qquad \left(  m,n\in \mathbb{N}_{0}:=\mathbb{N}
\cup \left \{  0\right \}  \right)
\end{equation}
where $h_{n}^{\lambda}$ is given by
\begin{equation}\label{gnorm}
h_{n}^{\lambda}=\frac{\left(  2\lambda \right)  _{n}\Gamma \left(  \lambda
+\frac{1}{2}\right)  \Gamma \left(  \frac{1}{2}\right)  }{n!\left(
n+\lambda \right)  \Gamma \left(  \lambda \right)  },
\end{equation}
$\delta_{n,m}$ is the Kronecker delta, and the Gamma function $\Gamma \left(  x\right)$ is defined by  \cite[p. 254, (6.1.1)]{26}
\begin{equation}\label{2}
\Gamma \left(  x\right)  =\int \limits_{0}^{\infty}t^{x-1}e^{-t}dt,\qquad  \Re \left(x\right)  >0.
\end{equation}
 The beta function is given by \cite[p. 258, (6.2.1)]{26}
\begin{equation*}
B\left(  a,b\right)  ={\displaystyle \int \limits_{0}^{1}}x^{a-1}\left(
 1-x\right)  ^{b-1}dx=\frac{\Gamma \left(  a\right)\Gamma \left(  b\right)  }{\Gamma \left(  a+b\right)  }, \qquad \Re\left(  a\right)  ,\Re\left(  b\right)  >0.
\end{equation*}

For our purposes, we shall also need to introduce the continuous Hahn polynomials \cite{27}
\begin{multline}\label{hahn}
p_{n}\left( x;a,b,c,d\right) \\
=i^{n}\frac{\left( a+c\right) _{n}\left(a+d\right) _{n}}{n!} \hyper{3}{2}{-n,\ n+a+b+c+d-1,\ a+ix}{a+c,\ a+d}{1}.
\end{multline}
which can also be written as a limiting case of the Wilson polynomials \cite{27}.

\subsection{Orthogonal polynomials on the unit ball}

Let $\left \Vert \boldsymbol{x}\right \Vert :=\left(  x_{1}^{2}+\cdots +x_{r}
^{2}\right)  ^{1/2}$ for $\boldsymbol{x=}\left(  x_{1},\dots,x_{r}\right)  \in
\mathbb{R}
^{r}$. The unit ball in $\mathbb{R}^{r}$ is denoted by $\mathbb{B}
^{r}:=\left \{  \boldsymbol{x}\in
\mathbb{R}^{r}:\left \Vert \boldsymbol{x}\right \Vert \leq1\right \} $. Let $W_{\mu}$ be
the weight function defined by%
\[
W_{\mu}\left(  \boldsymbol{x}\right)  =\left(  1-\left \Vert \boldsymbol{x}%
\right \Vert ^{2}\right)  ^{\mu-1/2},~~\mu>-1/2.
\]
We shall consider orthogonal polynomials on the unit ball, by considering the inner product%
\[
\left \langle f,g\right \rangle _{\mu}=
{\displaystyle \int \limits_{\mathbb{B}^{r}}}
W_{\mu}\left(  \boldsymbol{x}\right)  f\left(  \boldsymbol{x}\right)  g\left(
\boldsymbol{x}\right)  d\mathbf{x}
\]
where $d\mathbf{x}=dx_{1}\cdots dx_{r}$.

Let $\Pi^{r}$ denote the space of polynomials in $r$ real variables. Let $\Pi_{n}^{r}$ denote the linear space of polynomials in several variables of (total) degree at most $n$ for $n=0,1,2,\dots$. Let $\mathcal{V}_{n}^{r}\left(W_{\mu}\right)  $ be the space of orthogonal polynomials of total degree $n$ with respect to $W_{\mu}\left(  \boldsymbol{x}\right)$. Then $\dim \mathcal{V}_{n}^{r}\left(  W_{\mu}\right)  =\dbinom{n+r-1}{n}.$ The elements of the space $\mathcal{V}_{n}^{r}\left(  W_{\mu}\right)  $ are eigenfunctions
of a second order partial differential equation \cite[p.141, Eq. (5.2.3)]{29}
\begin{equation*}
 \sum \limits_{i=1}^{r}\frac{\partial^{2}P}{\partial x_{i}^{2}}-\sum \limits_{j=1}^{r}\frac{\partial}{\partial x_{j}}x_{j}\left[  2\mu -1+\sum \limits_{i=1}^{r}x_{i}\frac{\partial}{\partial x_{i}}\right]  P  =-\left(  n+r\right)  \left(  n+2\mu-1\right)  P.
\end{equation*}
The space $\mathcal{V}_{n}^{r}$ has several different bases. One orthogonal
basis of the space $\mathcal{V}_{n}^{r}$ can be expressed in terms of the Gegenbauer
polynomials \eqref{hyper} as \cite[p. 143]{29}
\begin{equation}
P_{\mathbf{n}}^{\mu}\left(  \boldsymbol{x}\right)  =\prod \limits_{j=1}^{r}\left(  1-\left \Vert \boldsymbol{x}_{j-1}\right \Vert ^{2}\right) ^{\frac{n_{j}}{2}}C_{n_{j}}^{\left(  \lambda_{j}\right)  }\left(  \frac{x_{j}}{\sqrt{1-\left \Vert \boldsymbol{x}_{j-1}\right \Vert ^{2}}}\right) \label{P},
\end{equation}
where $\lambda_{j}=\mu+\left \vert \mathbf{n}^{j+1}\right \vert +\frac{r-j}{2},$
and
\begin{equation}\label{notation}
\begin{cases}
\boldsymbol{x}_{0}   =0,\  \  \boldsymbol{x}_{j}=\left(  x_{1},\dots,x_{j} \right)  , \\
\mathbf{n}   =\left(  n_{1},\dots,n_{r}\right)  ,\  \text{\ }\left \vert
\mathbf{n}\right \vert =n_{1}+\dots+n_{r}=n,\\
\mathbf{n}^{j}   =\left(  n_{j},\dots,n_{r}\right)  ,\text{ \ }\left \vert \mathbf{n}^{j}\right \vert =n_{j}+\cdots+n_{r},\  \text{\  \ }1\leq j\leq
r,
\end{cases}
\end{equation}
and $\mathbf{n}^{r+1}:=0.$ More precisely,%
\[
\int \limits_{\mathbb{B}^{r}}W_{\mu}\left(  \boldsymbol{x}\right)P_{\mathbf{n}}^{\mu}\left(  \boldsymbol{x}\right)  P_{\mathbf{m}}^{\mu}\left(\boldsymbol{x}\right)  d\mathbf{x}=h_{\mathbf{n}}^{\mu}\delta_{\mathbf{n},\mathbf{m}},
\]
where $\delta_{\mathbf{n},\mathbf{m}}=\delta_{n_{1},m_{1}}\dots \delta
_{n_{r},m_{r}}$ and $h_{\mathbf{n}}^{\mu}$ is given by \cite{29}
\begin{equation}\label{Norm}
h_{\mathbf{n}}^{\mu}=\frac{\pi^{r/2}\Gamma \left(  \mu+\frac{1}{2}\right)
\left(  \mu+\frac{r}{2}\right)  _{\left \vert \mathbf{n}\right \vert }}{\Gamma \left(  \mu+\frac{r+1}{2}+\left \vert \mathbf{n}\right \vert \right)}\prod \limits_{j=1}^{r}\frac{\left(  \mu+\frac{r-j}{2}\right)  _{\left \vert
\mathbf{n}^{j}\right \vert }\left(  2\mu+2\left \vert \mathbf{n}^{j+1}\right \vert +r-j\right)  _{n_{j}}}{n_{j}!\left(  \mu+\frac{r-j+1}{2}\right)_{\left \vert \mathbf{n}^{j}\right \vert }}.
\end{equation}

\section{Main Results}\label{sec:ft}
Let us introduce
\begin{multline}\label{15}
f_{r}\left(  \mathbf{x};\mathbf{n},a,\mu \right)   :=f_{r}\left(
x_{1},\dots,x_{r};n_{1},\dots,n_{r},a,\mu \right)
\\ =\prod \limits_{j=1}^{r}\left(  1-\tanh^{2}x_{j}\right)  ^{a+\frac{r-j}{4}
}\ P_{\mathbf{n}}^{\mu}\left(  \upsilon_{1},\dots,\upsilon_{r}\right),
\end{multline}
for $r\geq1$, where $a,\mu$ are real parameters and
\begin{align*}
\upsilon_{1}\left(  x_{1}\right)   & =\upsilon_{1}=\tanh x_{1},\\
\upsilon_{r}\left(  x_{1},\dots,x_{r}\right)   & =\upsilon_{r}=\tanh x_{r}%
\sqrt{\left(  1-\tanh^{2}x_{1}\right)  \left(  1-\tanh^{2}x_{2}\right)
\cdots \left(  1-\tanh^{2}x_{r-1}\right)  },
\end{align*}
for $r\geq2$.

From the latter expression, we can write $f_{r}$ defined in \eqref{15} in terms of $f_{r-1}$ in the
following forms%
\begin{multline}\label{g1}
 f_{r}\left(  x_{1},\dots,x_{r};n_{1},\dots,n_{r},a,\mu \right)  \\
 =\left(  1-\tanh^{2}x_{1}\right)  ^{a+\frac{n_{2}+\cdots +n_{r}}{2}+\frac
{r-1}{4}}C_{n_{1}}^{\left(  n_{2}+\cdots +n_{r}+\mu+\frac{r-1}{2}\right)  }\left(
\tanh x_{1}\right)  \\
 \times f_{r-1}\left(  x_{2},\dots,x_{r};n_{2},\dots,n_{r},a,\mu \right),
\end{multline}
or%
\begin{multline}\label{g2}
 f_{r}\left(  x_{1},\dots,x_{r};n_{1},\dots,n_{r},a,\mu \right)  \\
 =\left(  1-\tanh^{2}x_{r}\right)  ^{a}C_{n_{r}}^{\left(  \mu \right)
}\left(  \tanh x_{r}\right)  \\
 \times f_{r-1}\left(  x_{1},\dots,x_{r-1};n_{1},\dots,n_{r-1},a+\frac{n_{r}}%
{2}+\frac{1}{4},\mu+n_{r}+\frac{1}{2}\right),
\end{multline}
for $r\geq1$ where the univariate Gegenbauer polynomials $C_{n}^{\left(  \lambda \right)  }\left(  x\right)  $ are defined in \eqref{3}. For $r=1$,
\[
f_{1}\left(  x_{1};n_{1},a,\mu \right)  =\left(  1-\tanh^{2}x_{1}\right)
^{a}C_{n_{1}}^{\left(  \mu \right)  }\left(  \tanh x_{1}\right).
\]

\subsection{The Fourier Transform of Orthogonal Polynomials on the Unit Ball}

The Fourier transform of a given univariate function $f(x)$ is defined by \cite[p.111, Eq. (7.1)]{3}%
\begin{equation}\label{16}
\mathcal{F}
\left(  f\left(  x\right)  \right)  =\int \limits_{-\infty}^{\infty}e^{-i\xi x}f\left(  x\right)  dx.
\end{equation}
In the $r$-variable case, the Fourier transform of a given multivariate function $f(x_{1},\dots,x_{r})$ is defined by \cite[p. 182, Eq. (11.1a)]{3}
\begin{equation}\label{17}
\mathcal{F}
\left(  f\left(  x_{1},\dots,x_{r}\right)  \right)  =\int \limits_{-\infty
}^{\infty}\cdots \int \limits_{-\infty}^{\infty}e^{-i\left(  \xi_{1}x_{1}%
+\cdots+\xi_{r}x_{r}\right)  }f\left(  x_{1},\dots,x_{r}\right)  dx_{1}\cdots dx_{r}.
\end{equation}

Next, we calculate the Fourier transform of the function $f_{r}\left(  \mathbf{x};\mathbf{n},a,\mu \right)  $ defined in \eqref{15} by using the induction method. In doing so, we first start with the following theorem.

\begin{theorem}\label{theorem31}
Let $f_{r}\left(  \mathbf{x};\mathbf{n},a,\mu \right)$ be defined in \eqref{15}. The following result holds true%
\begin{multline}\label{F1}
\mathcal{F}%
\left(  f_{r}\left(  \mathbf{x};\mathbf{n},a,\mu \right)  \right)     =\mathcal{F}%
\left(  f_{r}\left(  x_{1},\dots,x_{r};n_{1},\dots,n_{r},a,\mu \right)  \right) \\
  =\frac{2^{\left \vert \mathbf{n}^{2}\right \vert +2a+\frac{r-3}{2}}\left(
2\left(  \left \vert \mathbf{n}^{2}\right \vert +\mu+\frac{r-1}{2}\right)
\right)  _{n_{1}}}{n_{1}!} \\
  \times B\left(  a+\frac{\left \vert \mathbf{n}^{2}\right \vert +i\xi_{1}}%
{2}+\frac{r-1}{4},a+\frac{\left \vert \mathbf{n}^{2}\right \vert -i\xi_{1}}%
{2}+\frac{r-1}{4}\right)   \\
  \times \ _{3}F_{2}\left(
\genfrac{}{}{0pt}{0}{-n_{1},\ n_{1}+2\left(  \left \vert \mathbf{n}%
^{2}\right \vert +\mu+\frac{r-1}{2}\right)  ,\ a+\frac{\left \vert
\mathbf{n}^{2}\right \vert +i\xi_{1}}{2}+\frac{r-1}{4}}{\left \vert
\mathbf{n}^{2}\right \vert +2a+\frac{r-1}{2},\  \left \vert \mathbf{n}%
^{2}\right \vert +\mu+\frac{r}{2}} \mid1\right)   \\
\times \mathcal{F} \left(  f_{r-1}\left(  x_{2},\dots,x_{r};n_{2},\dots,n_{r},a,\mu \right)  \right),
\end{multline}
or
\begin{multline} \label{F2}
\mathcal{F}
\left(  f_{r}\left(  \mathbf{x};\mathbf{n},a,\mu \right)  \right)  =
\mathcal{F} \left(  f_{r}\left(  x_{1},\dots,x_{r};n_{1},\dots,n_{r},a,\mu \right)  \right) \\
=\frac{2^{2a-1}\left(  2\mu \right)  _{n_{r}}}{n_{r}!}B\left(  a+\frac{i\xi_{r}}{2},a-\frac{i\xi_{r}}{2}\right)  \\
\times \ _{3}F_{2}\left(\genfrac{}{}{0pt}{0}{-n_{r},\ n_{r}+2\mu,\ a+\frac{i\xi_{r}}{2}}{2a,\  \mu
+\frac{1}{2}}\mid1\right)   \\
\times \mathcal{F} \left(  f_{r-1}\left(  x_{1},\dots,x_{r-1};n_{1},\dots,n_{r-1},a+\frac{n_{r}}
{2}+\frac{1}{4},\mu+n_{r}+\frac{1}{2}\right)  \right)  .
\end{multline}

\end{theorem}

\begin{proof}
By using \eqref{g1}, the Fourier transform of the function $f_{r}$ defined in \eqref{15} can be calculated as follows by using relation \eqref{hyper}
\begin{multline*}
\mathcal{F} \left(  f_{r}\left(  x_{1},\dots,x_{r};n_{1},\dots,n_{r},a,\mu \right)  \right) \\
 =\int \limits_{-\infty}^{\infty}\cdots \int \limits_{-\infty}^{\infty
}e^{-i\left(  \xi_{1}x_{1}+\cdots+\xi_{r}x_{r}\right)  }\left(  1-\tanh^{2}%
x_{1}\right)  ^{a+\frac{n_{2}+\cdots+n_{r}}{2}+\frac{r-1}{4}}\\
  \times C_{n_{1}}^{\left(  \mu+n_{2}+\cdots+n_{r}+\frac{r-1}{2}\right)
}\left(  \tanh x_{1}\right)  f_{r-1}\left(  x_{2},\dots,x_{r};n_{2}%
,\dots,n_{r},a,\mu \right)  dx_{r}\cdots dx_{1}\\
  =\int \limits_{-\infty}^{\infty}e^{-i\xi_{1}x_{1}}\left(  1-\tanh^{2}%
x_{1}\right)  ^{a+\frac{n_{2}+\cdots +n_{r}}{2}+\frac{r-1}{4}}
   C_{n_{1}}^{\left(  \mu+n_{2}+\cdots +n_{r}+\frac{r-1}{2}\right)
}\left(  \tanh x_{1}\right)  dx_{1}\\
  \times \int \limits_{-\infty}^{\infty}\cdots \int \limits_{-\infty}^{\infty
}e^{-i\left(  \xi_{2}x_{2}+\cdots+\xi_{r}x_{r}\right)  }f_{r-1}\left(
x_{2},\dots,x_{r};n_{2},\dots,n_{r},a,\mu \right)  dx_{r}\cdots dx_{2}\\
  =\mathcal{F} \left(  f_{r-1}\left(  x_{2},\dots,x_{r};n_{2},\dots,n_{r},a,\mu \right)  \right)\\
  \times \int \limits_{-1}^{1}\left(  1+u\right)  ^{a+\frac{n_{2}%
+\cdots+n_{r}-i\xi_{1}}{2}+\frac{r-5}{4}}\left(  1-u\right)  ^{a+\frac
{n_{2}+\cdots +n_{r}+i\xi_{1}}{2}+\frac{r-5}{4}}
   C_{n_{1}}^{\left(  \mu+n_{2}+\cdots +n_{r}+\frac{r-1}{2}\right)
}\left(  u\right)  du \\
  =\frac{2^{n_{2}+\cdots+n_{r}+2a+\frac{r-3}{2}}\left(  2\left(  \mu
+n_{2}+\cdots+n_{r}+\frac{r-1}{2}\right)  \right)  _{n_{1}}}{n_{1}!}
 \mathcal{F} \left(  f_{r-1}\left(  x_{2},\dots,x_{r};n_{2},\dots,n_{r},a,\mu \right)  \right) \\
  \times \sum \limits_{l=0}^{n_{1}}\frac{\left(  -n_{1}\right)  _{l}\left(
n_{1}+2\left(  \mu+n_{2}+\cdots+n_{r}\right)  +r-1\right)  _{l}}{l!\left(
\mu+n_{2}+\cdots+n_{r}+\frac{r}{2}\right)  _{l}}\\
  \times \int \limits_{0}^{1}\left(  1-t\right)  ^{a+\frac{n_{2}+\cdots+n_{r}%
-i\xi_{1}}{2}+\frac{r-5}{4}}t^{a+\frac{n_{2}+\cdots+n_{r}+i\xi_{1}}{2}+\frac
{r-5}{4}+l}dt\\
  =\frac{2^{n_{2}+\cdots+n_{r}+2a+\frac{r-3}{2}}\left(  2\left(  \mu
+n_{2}+\cdots+n_{r}+\frac{r-1}{2}\right)  \right)  _{n_{1}}}{n_{1}!}\\
  \times%
\mathcal{F}%
\left(  f_{r-1}\left(  x_{2},\dots,x_{r};n_{2},\dots,n_{r},a,\mu \right)  \right)
\\
  \times B\left(  a+\frac{n_{2}+\cdots+n_{r}+i\xi_{1}}{2}+\frac{r-1}{4}%
,a+\frac{n_{2}+\cdots+n_{r}-i\xi_{1}}{2}+\frac{r-1}{4}\right) \\
  \times \ _{3}F_{2}\left(
\genfrac{}{}{0pt}{0}{-n_{1},n_{1}+2\left(  \mu+n_{2}+\cdots+n_{r}\right)
+r-1,a+\frac{n_{2}+\cdots+n_{r}+i\xi_{1}}{2}+\frac{r-1}{4}}{\mu+n_{2}%
+\cdots+n_{r}+\frac{r}{2},2a+n_{2}+\cdots+n_{r}+\frac{r-1}{2}}%
\mid1\right)  ,
\end{multline*}
which proves \eqref{F1}. Similarly, when we repeat this process by using the
equation (\ref{g2}), it follows

\begin{multline*}
\mathcal{F}
\left(  f_{r}\left(  x_{1},\dots,x_{r};n_{1},\dots,n_{r},a,\mu \right)  \right) \\
  =\int \limits_{-1}^{1}\left(  1-u\right)  ^{a+\frac{i\xi_{r}}{2}-1}\left(
1+u\right)  ^{a-\frac{i\xi_{r}}{2}-1}C_{n_{r}}^{\left(  \mu \right)  }\left(
u\right)  du\\
\times \mathcal{F} \left(  f_{r-1}\left(  x_{1},\dots,x_{r-1};n_{1},\dots,n_{r-1},a+\frac{n_{r}}{2}+\frac{1}{4},\mu+n_{r}+\frac{1}{2}\right)  \right) \\
  =\frac{2^{2a-1}\left(  2\mu \right)  _{n_{r}}}{n_{r}!}B\left(  a+\frac
{i\xi_{r}}{2},a-\frac{i\xi_{r}}{2}\right)  \ _{3}F_{2}\left( \genfrac{}{}{0pt}{0}{-n_{r},\ n_{r}+2\mu,\ a+\frac{i\xi_{r}}{2}}{\mu+1/2,\ 2a} \mid1\right) \\
  \times
\mathcal{F}
\left(  f_{r-1}\left(  x_{1},\dots,x_{r-1};n_{1},\dots,n_{r-1},a+\frac{n_{r}} {2}+\frac{1}{4},\mu+n_{r}+\frac{1}{2}\right)  \right)  .
\end{multline*}

\end{proof}

By applying Theorem \ref{theorem31} consecutively, we can give the next theorem.

\begin{theorem}
The Fourier transform of the function $f_{r}\left(  \mathbf{x};\mathbf{n} ,a,\mu \right)$ defined in \eqref{15} is explicitly given as follows%
\begin{multline} \label{18}
\mathcal{F}
\left(  f_{r}\left(  \mathbf{x};\mathbf{n},a,\mu \right)  \right)  =
\mathcal{F}
\left(  f_{r}\left(  x_{1},\dots,x_{r};n_{1},\dots,n_{r},a,\mu \right)  \right) \\
=2^{2ra+\frac{r\left(  r-5\right)  }{4}+
{\textstyle \sum \limits_{j=1}^{r-1}} {jn}_{j+1}}\prod \limits_{j=1}^{r}\left \{  \frac{\left(  2\left(
\left \vert \mathbf{n}^{j+1}\right \vert +\mu+\frac{r-j}{2}\right)  \right)
_{n_{j}}}{n_{j}!}\Theta_{j}^{r}\left(  a,\mu,\mathbf{n};\xi_{j}\right)
\right \}  ,
\end{multline}
where
\begin{multline*}
\Theta_{j}^{r}\left(  a,\mu,\mathbf{n};\xi_{j}\right)  =B\left(
a+\frac{\left \vert \mathbf{n}^{j+1}\right \vert +i\xi_{j}}{2}+\frac{r-j}
{4},a+\frac{\left \vert \mathbf{n}^{j+1}\right \vert -i\xi_{j}}{2}+\frac{r-j}
{4}\right)  \\
\times \ _{3}F_{2}\left(
\genfrac{}{}{0pt}{0}{-n_{j},\ n_{j}+2\left(  \left \vert \mathbf{n}%
^{j+1}\right \vert +\mu+\frac{r-j}{2}\right)  ,\ a+\frac{\left \vert
\mathbf{n}^{j+1}\right \vert +i\xi_{j}}{2}+\frac{r-j}{4}}{\left \vert
\mathbf{n}^{j+1}\right \vert +\mu+\frac{r-j+1}{2},\  \left \vert \mathbf{n}%
^{j+1}\right \vert +2a+\frac{r-j}{2}}%
\mid1\right)  ,
\end{multline*}
which can be also expressed in terms of the continuous Hahn polynomials defined in \eqref{hahn}%
\begin{multline*}
\Theta_{j}^{r}\left(  a,\mu,\mathbf{n};\xi_{j}\right)  =\frac{n_{j}!}%
{i^{n_{j}}\left(  \left \vert \mathbf{n}^{j+1}\right \vert +\mu+\frac{r-j+1}%
{2}\right)  _{n_{j}}\left(  \left \vert \mathbf{n}^{j+1}\right \vert
+2a+\frac{r-j}{2}\right)  _{n_{j}}}   \\
\times B\left(  a+\frac{\left \vert \mathbf{n}^{j+1}\right \vert +i\xi_{j}}%
{2}+\frac{r-j}{4},a+\frac{\left \vert \mathbf{n}^{j+1}\right \vert -i\xi_{j}}%
{2}+\frac{r-j}{4}\right)    \\
\times p_{n_{j}}\left(  \frac{\xi_{j}}{2};a+\frac{\left \vert \mathbf{n}%
^{j+1}\right \vert }{2}+\frac{r-j}{4},\mu-a+\frac{\left \vert \mathbf{n}%
^{j+1}\right \vert +1}{2}+\frac{r-j}{4}\right.    \\
\left.  ,\mu-a+\frac{\left \vert \mathbf{n}^{j+1}\right \vert +1}{2}+\frac
{r-j}{4},a+\frac{\left \vert \mathbf{n}^{j+1}\right \vert }{2}+\frac{r-j}%
{4}\right)  .
\end{multline*}

\end{theorem}

\begin{proof}
The proof follows by induction on $r$ by applying Theorem \ref{theorem31} successively. For $r=1$ the Fourier transform of
\begin{equation}\label{g1dim}
f_{1}\left(  x_{1};n_{1},a,\mu \right)  =\left(  1-\tanh^{2}x_{1}\right)
^{a}\ C_{n_{1}}^{\left(  \mu \right)  }\left(  \tanh x_{1}\right)
\end{equation}
follows from \eqref{hyper} (see \cite{7})%
\begin{multline}\label{gF1dim}
\mathcal{F}%
\left(  f_{1}\left(  x_{1};n_{1},a,\mu \right)  \right)  =\int \limits_{-\infty}^{\infty}e^{-i\xi_{1}x_{1}}\left(  1-\tanh^{2}x_{1}\right) ^{a}\ C_{n_{1}}^{\left(  \mu \right)  }\left(  \tanh x_{1}\right) dx_{1} \\
 =\frac{2^{2a-1}\left(  2\mu \right)  _{n_{1}}}{n_{1}!}\Theta_{1}^{1}\left( a,\mu,n_{1};\xi_{1}\right)  ,
\end{multline}
where
\[
\Theta_{1}^{1}\left(  a,\mu,n_{1};\xi_{1}\right)  =\ _{3}F_{2}\left(
\genfrac{}{}{0pt}{0}{-n_{1},\ n_{1}+2\mu,\ a+\frac{i\xi_{1}}{2}}{2a,\  \mu+1/2}%
\mid1\right)  B\left(  a+\frac{i\xi_{1}}{2},\ a-\frac{i\xi_{1}}{2}\right).
\]
It can be rewritten \cite{7} in terms of the continuous Hahn polynomials $p_{n}\left(  x;a,b,c,d\right)$ as
\begin{multline*}%
\mathcal{F}%
\left(  f_{1}\left(  x_{1};n_{1},a,\mu \right)  \right)  =\frac {2^{2a-1}\left(  2\mu \right)  _{n_{1}}}{i^{n_{1}}\left(  2a\right)  _{n_{1} }\left(  \mu+1/2\right)  _{n_{1}}}B\left(  a+\frac{i\xi_{1}}{2},\ a-\frac
{i\xi_{1}}{2}\right)  \\
\times p_{n_{1}}\left(  \frac{\xi_{1}}{2};a,\mu-a+1/2,\mu-a+1/2,a\right)  .
\end{multline*}
For the case $r=2,$ in view of \eqref{g1} we can write
\[
f_{2}\left(  x_{1},x_{2};n_{1},n_{2},a,\mu \right)  =\left(  1-\tanh^{2}%
x_{1}\right)  ^{a+\frac{n_{2}}{2}+\frac{1}{4}}C_{n_{1}}^{\left(  n_{2}%
+\mu+\frac{1}{2}\right)  }\left(  \tanh x_{1}\right)  f_{1}\left(  x_{2}%
;n_{2},a,\mu \right).
\]
By using now \eqref{F1} it yields
\begin{multline*}
\mathcal{F}%
\left(  f_{2}\left(  x_{1},x_{2};n_{1},n_{2},a,\mu \right)  \right)
=\frac{2^{n_{2}+2a-\frac{1}{2}}\left(  2\left(  n_{2}+\mu+\frac{1}{2}\right)
\right)  _{n_{1}}}{n_{1}!}
\mathcal{F}%
\left(  f_{1}\left(  x_{2};n_{2},a,\mu \right)  \right)  \\
\times B\left(  a+\frac{n_{2}+i\xi_{1}}{2}+\frac{1}{4},a+\frac{n_{2}-i\xi_{1}%
}{2}+\frac{1}{4}\right)  \\
\times \ _{3}F_{2}\left(
\genfrac{}{}{0pt}{0}{-n_{1},\ n_{1}+2\left(  n_{2}+\mu+\frac{1}{2}\right)
,\ a+\frac{n_{2}+i\xi_{1}}{2}+\frac{1}{4}}{n_{2}+2a+\frac{1}{2},\ n_{2}+\mu+1}%
\mid1\right)  .
\end{multline*}
{}From \eqref{gF1dim}, we can write%
\begin{multline} \label{biv-fourier}
\mathcal{F} \left(  f_{2}\left(  x_{1},x_{2};n_{1},n_{2},a,\mu \right)  \right) \\
  =\frac{2^{n_{2}+4a-\frac{3}{2}}\left(  2\mu \right)  _{n_{2}}\left(
2\left(  n_{2}+\mu+\frac{1}{2}\right)  \right)  _{n_{1}}}{n_{1}!n_{2}%
!}B\left(  a+\frac{n_{2}+i\xi_{1}}{2}+\frac{1}{4},a+\frac{n_{2}-i\xi_{1}}%
{2}+\frac{1}{4}\right) \\
\times~_{3}F_{2}\left(
\genfrac{}{}{0pt}{0}{-n_{1},\ n_{1}+2\left(  n_{2}+\mu+\frac{1}{2}\right)
,\ a+\frac{n_{2}+i\xi_{1}}{2}+\frac{1}{4}}{\genfrac{}{}{0pt}{0}{{}%
}{n_{2}+2a+\frac{1}{2},\ n_{2}+\mu+1}}%
\mid1\right)  \\
\times B\left(  a+\frac{i\xi_{2}}{2},\ a-\frac{i\xi_{2}}{2}\right)
\ _{3}F_{2}\left(
\genfrac{}{}{0pt}{0}{-n_{2},\ n_{2}+2\mu,\ a+\frac{i\xi_{2}}{2}}{2a,\  \mu+1/2}%
\mid1\right)  \\
  =\frac{2^{n_{2}+4a-\frac{3}{2}}\left(  2\mu \right)  _{n_{2}}\left(
2\left(  n_{2}+\mu+\frac{1}{2}\right)  \right)  _{n_{1}}}{n_{1}!n_{2}!}%
\Theta_{1}^{2}\left(  a,\mu,n_{1},n_{2};\xi_{1}\right)  \Theta_{2}^{2}\left(
a,\mu,n_{1},n_{2};\xi_{2}\right)
\end{multline}
where%
\begin{multline*}
\Theta_{1}^{2}\left(  a,\mu,n_{1},n_{2};\xi_{1}\right)  =B\left(
a+\frac{n_{2}+i\xi_{1}}{2}+\frac{1}{4},a+\frac{n_{2}-i\xi_{1}}{2}+\frac{1}%
{4}\right)  \\
 \times_{3}F_{2}\left( \genfrac{}{}{0pt}{0}{-n_{1},\ n_{1}+2\left(  n_{2}+\mu+\frac{1}{2}\right)
,\ a+\frac{n_{2}+i\xi_{1}}{2}+\frac{1}{4}}{n_{2}+2a+\frac{1}{2},\ n_{2}+\mu+1}%
\mid1\right)
\end{multline*}
and%
\[
\Theta_{2}^{2}\left(  a,\mu,n_{1},n_{2};\xi_{2}\right)  =B\left(  a+\frac
{i\xi_{2}}{2},\ a-\frac{i\xi_{2}}{2}\right)  \ _{3}F_{2}\left(
\genfrac{}{}{0pt}{0}{-n_{2},\ n_{2}+2\mu,\ a+\frac{i\xi_{2}}{2}}{2a,\  \mu+1/2}
\mid1\right)  .
\]
Both expressions can be written again in terms of the continuous Hahn polynomials \eqref{hahn} as
\begin{multline*}
\Theta_{1}^{2}\left(  a,\mu,n_{1},n_{2};\xi_{1}\right)  =\dfrac{n_{1}%
!}{i^{n_{1}}\left(  n_{2}+\mu+1\right)  _{n_{1}}\left(  n_{2}+2a+\frac{1}%
{2}\right)  _{n_{1}}} \\
 \times B\left(  a+\frac{n_{2}+i\xi_{1}}{2}+\frac{1}{4},a+\frac{n_{2}-i\xi_{1}%
}{2}+\frac{1}{4}\right)  \\
\times p_{n_{1}}\left(  \frac{\xi_{1}}{2};a+\frac{n_{2}}{2}+\frac{1}{4}%
,\mu-a+\frac{2n_{2}+3}{4},\mu-a+\frac{2n_{2}+3}{4},a+\frac{n_{2}}{2}+\frac
{1}{4}\right)
\end{multline*}
and%
\begin{multline*}
\Theta_{2}^{2}\left(  a,\mu,n_{1},n_{2};\xi_{2}\right)  =\frac{n_{2}%
!}{i^{n_{2}}\left(  \mu+\frac{1}{2}\right)  _{n_{2}}\left(  2a\right)
_{n_{2}}}B\left(  a+\frac{i\xi_{2}}{2},a-\frac{i\xi_{2}}{2}\right)    \\
\times p_{n_{2}}\left(  \frac{\xi_{2}}{2};a,\mu-a+\frac{1}{2},\mu-a+\frac
{1}{2},a\right)  .
\end{multline*}
The proof follows now by induction on $r$.
\end{proof}

\subsection{The class of special functions using Fourier transform of the orthogonal polynomials on the unit ball}

The Parseval identity corresponding to \eqref{16} is given by \cite[p.118, Eq. (7.17)]{3}%
\begin{equation}\label{monoparseval}
\int \limits_{-\infty}^{\infty}f\left(  x\right)  \overline{g\left(  x\right)
}dx=\frac{1}{2\pi}\int \limits_{-\infty}^{\infty}%
\mathcal{F}%
\left(  f\left(  x\right)  \right)  \overline{%
\mathcal{F}%
\left(  g\left(  x\right)  \right)  }d\xi,%
\end{equation}
and in $r$-variable case, Parseval's identity corresponding to \eqref{17}
is \cite[p. 183, (iv)]{3}%
\begin{multline}\label{multi}
 \int \limits_{-\infty}^{\infty}\cdots \int \limits_{-\infty}^{\infty}f\left(
x_{1},\dots,x_{r}\right)  \overline{g\left(  x_{1},\dots,x_{r}\right)  }%
dx_{1}\cdots dx_{r} \\
 =\frac{1}{\left(  2\pi \right)  ^{r}}\int \limits_{-\infty}^{\infty} \cdots \int \limits_{-\infty}^{\infty}%
\mathcal{F} \left(  f\left(  x_{1},\dots,x_{r}\right)  \right)  \overline{ \mathcal{F} \left(  g\left(  x_{1},\dots,x_{r}\right)  \right)  }d\xi_{1}\cdots d\xi_{r}.
\end{multline}

\begin{theorem}
Let $\mathbf{n}$ and $\mathbf{n}^{j}$ be defined as in \eqref{notation}, let
$\mathbf{a=}\left(  a_{1},a_{2}\right)  $ and $\left \vert \mathbf{a}%
\right \vert =a_{1}+a_{2}$. Then, the following equality is satisfied%
\begin{multline*}
 \int \limits_{-\infty}^{\infty}\cdots \int \limits_{-\infty}^{\infty}%
\ _{r}D_{\boldsymbol{n}}\left(  i\boldsymbol{x};a_{1},a_{2}\right)
\ _{r}D_{\boldsymbol{m}}\left(  -i\boldsymbol{x};a_{2},a_{1}\right)
\mathbf{dx}
 =\left(  2\pi \right)  ^{r}2^{-2r\left \vert \mathbf{a}\right \vert
+r+1}h_{\mathbf{n}}^{\left(  a_{1}+a_{2}-\frac{1}{2}\right)  }\\
 \times \prod \limits_{j=1}^{r}\frac{\left(  n_{j}!\right)  ^{2}\Gamma \left(
\left \vert \mathbf{n}^{j+1}\right \vert +2a_{1}+\frac{r-j}{2}\right)
\Gamma \left(  \left \vert \mathbf{n}^{j+1}\right \vert +2a_{2}+\frac{r-j}%
{2}\right)  }{2^{2\left \vert \mathbf{n}^{j+1}\right \vert }\left(  \left(
2\left \vert \mathbf{n}^{j+1}\right \vert +2\left \vert \mathbf{a}\right \vert
+r-j-1\right)  _{n_{j}}\right)  ^{2}}\delta_{n_{j},m_{j}},%
\end{multline*}
for $a_{1},a_{2}>0$ where $h_{\mathbf{n}}^{\left(  a_{1}+a_{2}-\frac{1}%
{2}\right)  }$ is given in \eqref{Norm} and%
\begin{multline*}
 \ _{r}D_{\mathbf{n}}\left(  \boldsymbol{x};a_{1},a_{2}\right)
=\prod \limits_{j=1}^{r}\left \{  \Gamma \left(  a_{1}+\frac{\left \vert
\mathbf{n}^{j+1}\right \vert -x_{j}}{2}+\frac{r-j}{4}\right)  \Gamma \left(
a_{1}+\frac{\left \vert \mathbf{n}^{j+1}\right \vert +x_{j}}{2}+\frac{r-j}%
{4}\right)  \right. \\
 \times \left.  _{3}F_{2}\left(
\genfrac{}{}{0pt}{0}{-n_{j},\ n_{j}+2\left(  \left \vert \mathbf{n}%
^{j+1}\right \vert +\left \vert \mathbf{a}\right \vert +\frac{r-j-1}{2}\right)
,\ a_{1}+\frac{\left \vert \mathbf{n}^{j+1}\right \vert +x_{j}}{2}+\frac{r-j}%
{4}}{\left \vert \mathbf{n}^{j+1}\right \vert +\left \vert \mathbf{a}\right \vert
+\frac{r-j}{2},\  \left \vert \mathbf{n}^{j+1}\right \vert +2a_{1}+\frac{r-j}{2}}%
\mid1\right)  \right \},
\end{multline*}
which can be expressed in terms of the continuous Hahn polynomials \eqref{hahn} by
\begin{multline*}
\ _{r}D_{\mathbf{n}}\left(  \boldsymbol{x};a_{1},a_{2}\right)  =%
{\displaystyle \prod \limits_{j=1}^{r}}
\left \{  \dfrac{n_{j}!i^{-n_{j}}}{\left(  \left \vert \mathbf{n}^{j+1}%
\right \vert +2a_{1}+\dfrac{r-j}{2}\right)  _{n_{j}}\left(  \left \vert
\mathbf{n}^{j+1}\right \vert +\left \vert \mathbf{a}\right \vert +\dfrac{r-j}%
{2}\right)  _{n_{j}}}\right.  \\
\times \left.  \Gamma \left(  a_{1}+\dfrac{\left \vert \mathbf{n}^{j+1}%
\right \vert -x_{j}}{2}+\dfrac{r-j}{4}\right)  \Gamma \left(  a_{1}%
+\dfrac{\left \vert \mathbf{n}^{j+1}\right \vert +x_{j}}{2}+\dfrac{r-j}%
{4}\right)  \right.  \\
\times p_{n_{j}}\left(  -\dfrac{ix_{j}}{2};a_{1}+\dfrac{\left \vert
\mathbf{n}^{j+1}\right \vert }{2}+\dfrac{r-j}{4},a_{2}+\dfrac{\left \vert
\mathbf{n}^{j+1}\right \vert }{2}+\dfrac{r-j}{4}\right.  \\
 \left.  \left.  ,a_{2}+\dfrac{\left \vert \mathbf{n}^{j+1}\right \vert
}{2}+\dfrac{r-j}{4},a_{1}+\dfrac{\left \vert \mathbf{n}^{j+1}\right \vert }%
{2}+\dfrac{r-j}{4}\right)  \right \},
\end{multline*}
for $r\geq1$.
\end{theorem}

\begin{proof}
The proof follows by using induction on $r$. For $r=1$ we get the specific functions from \eqref{g1dim}%
\begin{equation}
\begin{cases}
f_{1}\left(  x_{1};n_{1},a_{1},\mu_{1}\right)   =\left(  1-\tanh^{2}%
x_{1}\right)  ^{a_{1}}P_{n_{1}}^{\mu_{1}}\left(  \upsilon_{1}\right)  =\left(
1-\tanh^{2}x_{1}\right)  ^{a_{1}}\ C_{n_{1}}^{\left(  \mu_{1}\right)  }\left(
\tanh x_{1}\right)  ,\label{spec}\\
g_{1}\left(  x_{1};m_{1},a_{2},\mu_{2}\right)    =\left(  1-\tanh^{2}%
x_{1}\right)  ^{a_{2}}P_{m_{1}}^{\mu_{2}}\left(  \upsilon_{1}\right)  =\left(
1-\tanh^{2}x_{1}\right)  ^{a_{2}}\ C_{m_{1}}^{\left(  \mu_{2}\right)  }\left(
\tanh x_{1}\right) ,
\end{cases}
\end{equation}
where $\upsilon_{1}=\tanh x_{1}.$ According to \eqref{spec} and \eqref{gF1dim}, we use Parseval's identity to obtain
\begin{multline*}
 2\pi {\displaystyle \int \limits_{-\infty}^{\infty}}
\left(  1-\tanh^{2}x_{1}\right)  ^{a_{1}+a_{2}}\ C_{n_{1}}^{\left(  \mu
_{1}\right)  }\left(  \tanh x_{1}\right)  C_{m_{1}}^{\left(  \mu_{2}\right)
}\left(  \tanh x_{1}\right)  dx_{1}\\
 =2\pi {\displaystyle \int \limits_{-1}^{1}}
\left(  1-u^{2}\right)  ^{a_{1}+a_{2}-1}C_{n_{1}}^{\left(  \mu_{1}\right)
}\left(  u\right)  C_{m_{1}}^{\left(  \mu_{2}\right)  }\left(  u\right)  du
 =\frac{2^{2\left(  a_{1}+a_{2}-1\right)  }\left(  2\mu_{1}\right)  _{n_{1}%
}\left(  2\mu_{2}\right)  _{m_{1}}}{n_{1}!m_{1}!\Gamma \left(  2a_{1}\right)
\Gamma \left(  2a_{2}\right)  }\\
 \times
{\displaystyle \int \limits_{-\infty}^{\infty}}
\text{ }\Gamma \left(  a_{1}+\frac{i\xi_{1}}{2}\right)  \Gamma \left(
a_{1}-\frac{i\xi_{1}}{2}\right)  ~\overline{\Gamma \left(  a_{2}+\frac{i\xi
_{1}}{2}\right)  \Gamma \left(  a_{2}-\frac{i\xi_{1}}{2}\right)  }\\
 \times~_{3}F_{2}\left(
\genfrac{}{}{0pt}{0}{-n_{1},\ n_{1}+2\mu_{1},\ a_{1}+\frac{i\xi_{1}}%
{2}}{2a_{1},\  \mu_{1}+1/2}%
\mid1\right)  ~\overline{_{3}F_{2}\left(
\genfrac{}{}{0pt}{0}{-m_{1},\ m_{1}+2\mu_{2},\ a_{2}+\frac{i\xi_{1}}%
{2}}{2a_{2},\  \mu_{2}+1/2}%
\mid1\right)  }d\xi_{1}.
\end{multline*}
By assuming
\[
\mu_{1}=\mu_{2}=a_{1}+a_{2}-\frac{1}{2},
\]
and considering the orthogonality relation \eqref{ort}, we obtain that the special function
\begin{multline*}
\ _{1}D_{n_{1}}\left(  x_{1};a_{1},a_{2}\right)    \\
=  \Gamma \left(
a_{1}-\frac{x_{1}}{2}\right)  \Gamma \left(  a_{1}+\frac{x_{1}}{2}\right)
~_{3}F_{2}\left(
\genfrac{}{}{0pt}{0}{-n_{1},\ n_{1}+2\left(  a_{1}+a_{2}\right)
-1,\ a_{1}+\frac{x_{1}}{2}}{a_{1}+a_{2},\ 2a_{1}}%
\mid1\right) \\
 =\frac{n_{1}!i^{-n_{1}}}{\left(  2a_{1}\right)  _{n_{1}}\left(  a_{1}%
+a_{2}\right)  _{n_{1}}}\Gamma \left(  a_{1}-\frac{x_{1}}{2}\right)
\Gamma \left(  a_{1}+\frac{x_{1}}{2}\right)  p_{n_{1}}\left(  \frac{-ix_{1}}%
{2};a_{1},a_{2},a_{2},a_{1}\right)
\end{multline*}
has the orthogonality relation%
\begin{multline*}
 \int \limits_{-\infty}^{\infty}\ _{1}D_{n_{1}}\left(  ix_{1};a_{1}%
,a_{2}\right)  \ _{1}D_{m_{1}}\left(  -ix_{1};a_{2},a_{1}\right)   dx_{1}\\
 =\frac{2\pi n_{1}!\Gamma \left(  2a_{1}\right)  \Gamma \left(  2a_{2}\right)
\Gamma^{2}\left(  a_{1}+a_{2}\right)  }{\left(  n_{1}+a_{1}+a_{2}-\frac{1}%
{2}\right)  \Gamma \left(  2a_{1}+2a_{2}+n_{1}-1\right)  }\delta_{n_{1},m_{1}%
}\\
 =\frac{2\pi \left(  n_{1}!\right)  ^{2}\Gamma \left(  2a_{1}\right)
\Gamma \left(  2a_{2}\right)  }{2^{2\left(  a_{1}+a_{2}-1\right)  }\left(
\left(  2a_{1}+2a_{2}-1\right)  _{n_{1}}\right)  ^{2}}h_{n_{1}}^{\left(
a_{1}+a_{2}-\frac{1}{2}\right)  }\delta_{n_{1},m_{1}},
\end{multline*}
where $h_{n_{1}}^{\left(  a_{1}+a_{2}-\frac{1}{2}\right)  }$ is given in \eqref{gnorm}. As a consequence it follows%
\begin{multline*}
 \int \limits_{-\infty}^{\infty}\Gamma \left(  a_{1}+ix_{1}\right)
\Gamma \left(  a_{1}-ix_{1}\right)  \Gamma \left(  a_{2}-ix_{1}\right)
\Gamma \left(  a_{2}+ix_{1}\right) \\
 \times p_{n_{1}}\left(  x_{1};a_{1},a_{2},a_{2},a_{1}\right)  p_{m_{1}
}\left(  x_{1};a_{1},a_{2},a_{2},a_{1}\right)  dx_{1}\\
 =\frac{\pi \Gamma \left(  2a_{1}+n_{1}\right)  \Gamma \left(  2a_{2}
+n_{1}\right)  \Gamma^{2}\left(  a_{1}+a_{2}+n_{1}\right)  }{n_{1}!\left(
n_{1}+a_{1}+a_{2}-\frac{1}{2}\right)  \Gamma \left(  2a_{1}+2a_{2}
+n_{1}-1\right)  }\delta_{n_{1},m_{1}},
\end{multline*}
for $a_{1},a_{2}>0$ which gives the orthogonality relation for continuous Hahn
polynomials $p_{n_{1}}\left(  x_{1};a_{1},a_{2},a_{2},a_{1}\right)  $, which
was proved by Koelink \cite{7}.

For $r=2,$ we consider the specific functions from \eqref{15}%
\begin{equation}\label{bispec}
\begin{cases}
f_{2}\left(  x_{1},x_{2};n_{1},n_{2},a_{1},\mu_{1}\right)   & =\left(
1-\tanh^{2}x_{1}\right)  ^{a_{1}+\frac{1}{4}}\  \left(  1-\tanh^{2}%
x_{2}\right)  ^{a_{1}}\ P_{n_{1},n_{2}}^{\mu_{1}}\left(  \upsilon_{1}%
,\upsilon_{2}\right)  ,\\
g_{2}\left(  x_{1},x_{2};m_{1},m_{2},a_{2},\mu_{2}\right)   & =\left(
1-\tanh^{2}x_{1}\right)  ^{a_{2}+\frac{1}{4}}\  \left(  1-\tanh^{2}%
x_{2}\right)  ^{a_{2}}\ P_{m_{1},m_{2}}^{\mu_{2}}\left(  \upsilon_{1}%
,\upsilon_{2}\right) ,
\end{cases}
\end{equation}
where $\upsilon_{1}=\tanh x_{1}$ and $\upsilon_{2}=\tanh x_{2}\sqrt
{1-\tanh^{2}x_{1}}.$ According to \eqref{bispec} and \eqref{biv-fourier}, if we use Parseval's identity again and apply the transforms $\tanh x_{1}=u,$ $\tanh x_{2}=\frac{v}{\sqrt{1-u^{2}}}$, we obtain
\begin{multline*}
 {\displaystyle \int \limits_{-\infty}^{\infty}} {\displaystyle \int \limits_{-\infty}^{\infty}}
\left(  1-\tanh^{2}x_{1}\right)  ^{a_{1}+a_{2}+\frac{1}{2}}\  \left(1-\tanh^{2}x_{2}\right)  ^{a_{1}+a_{2}}\ P_{n_{1},n_{2}}^{\mu_{1}}\left(\upsilon_{1},\upsilon_{2}\right)  P_{m_{1},m_{2}}^{\mu_{2}}\left(  \upsilon_{1},\upsilon_{2}\right) dx_{1}dx_{2}\\
 ={\displaystyle \int \limits_{-1}^{1}} {\displaystyle \int \limits_{-\sqrt{1-u^{2}}}^{\sqrt{1-u^{2}}}}
\left(  1-u^{2}-v^{2}\right)  ^{a_{1}+a_{2}-1}P_{n_{1},n_{2}}^{\mu_{1}}\left(u,v\right)  P_{m_{1},m_{2}}^{\mu_{2}}\left(  u,v\right)  dvdu\\
 =\frac{2^{n_{2}+m_{2}+4a_{1}+4a_{2}-3}\left(  2\mu_{1}\right)  _{n_{2}%
}\left(  2\mu_{2}\right)  _{m_{2}}\left(  2\left(  n_{2}+\mu_{1}+\frac{1}%
{2}\right)  \right)  _{n_{1}}\left(  2\left(  m_{2}+\mu_{2}+\frac{1}%
{2}\right)  \right)  _{m_{1}}}{ 4 \pi^{2} n_{1}!n_{2}!m_{1}!m_{2}!}\\
 \times%
{\displaystyle \int \limits_{-\infty}^{\infty}}
{\displaystyle \int \limits_{-\infty}^{\infty}}
B\left(  a_{1}+\frac{n_{2}+i\xi_{1}}{2}+\frac{1}{4},a_{1}+\frac{n_{2}-i\xi
_{1}}{2}+\frac{1}{4}\right)  B\left(  a_{1}+\frac{i\xi_{2}}{2},\ a_{1}%
-\frac{i\xi_{2}}{2}\right) \\
 \times~_{3}F_{2}\left(
\genfrac{}{}{0pt}{0}{-n_{1},\ n_{1}+2\left(  n_{2}+\mu_{1}+\frac{1}{2}\right)
,\ a_{1}+\frac{n_{2}+i\xi_{1}}{2}+\frac{1}{4}}{n_{2}+2a_{1}+\frac{1}%
{2},\ n_{2}+\mu_{1}+1}%
\mid1\right)  \\ \times \ _{3}F_{2}\left(
\genfrac{}{}{0pt}{0}{-n_{2},\ n_{2}+2\mu_{1},\ a_{1}+\frac{i\xi_{2}}%
{2}}{2a_{1},\  \mu_{1}+1/2}%
\mid1\right) \\
 \times \overline{B\left(  a_{2}+\frac{m_{2}+i\xi_{1}}{2}+\frac{1}{4}%
,a_{2}+\frac{m_{2}-i\xi_{1}}{2}+\frac{1}{4}\right)  B\left(  a_{2}+\frac
{i\xi_{2}}{2},\ a_{2}-\frac{i\xi_{2}}{2}\right)  }\\
 \times \overline{~_{3}F_{2}\left(
\genfrac{}{}{0pt}{0}{-m_{1},\ m_{1}+2\left(  m_{2}+\mu_{2}+\frac{1}{2}\right)
,\ a_{2}+\frac{m_{2}+i\xi_{1}}{2}+\frac{1}{4}}{m_{2}+2a_{2}+\frac{1}%
{2},\ m_{2}+\mu_{2}+1}%
\mid1\right)  }\\
 \times~\overline{_{3}F_{2}\left(
\genfrac{}{}{0pt}{0}{-m_{2},\ m_{2}+2\mu_{2},\ a_{2}+\frac{i\xi_{2}}%
{2}}{2a_{2},\  \mu_{2}+1/2}%
\mid1\right)  }d\xi_{1}d\xi_{2}.
\end{multline*}
If we fix%
\[
\mu_{1}=\mu_{2}=a_{1}+a_{2}-\frac{1}{2}
\]
and use the orthogonality relation \eqref{ort} it yields
\begin{multline*}
 \int \limits_{-\infty}^{\infty}\int \limits_{-\infty}^{\infty}~_{2}%
D_{n_{1},n_{2}}\left(  ix_{1},ix_{2};a_{1},a_{2}\right)  ~_{2}D_{m_{1},m_{2}%
}\left(  -ix_{1},-ix_{2};a_{2},a_{1}\right)  dx_{1}dx_{2}\\
 =\frac{\Gamma \left(  2a_{1}\right)  \Gamma \left(  2a_{2}\right)
\Gamma \left(  2a_{1}+n_{2}+\frac{1}{2}\right)  \Gamma \left(  2a_{2}%
+n_{2}+\frac{1}{2}\right)  4\pi^{2}\left(  n_{1}!\right)  ^{2}\left(
n_{2}!\right)  ^{2}}{2^{2n_{2}+4a_{1}+4a_{2}-3}\left(  2\left(  n_{2}%
+a_{1}+a_{2}\right)  \right)  _{n_{1}}^{2}\left(  2a_{1}+2a_{2}-1\right)
_{n_{2}}^{2}}\\
 \times h_{n_{1},n_{2}}^{\left(  a_{1}+a_{2}-\frac{1}{2}\right)  }%
\delta_{n_{1},m_{1}}\delta_{n_{2},m_{2}},
\end{multline*}
where $h_{n_{1},n_{2}}^{\left(  a_{1}+a_{2}-\frac{1}{2}\right)  }$ is given in \eqref{Norm} and%
\begin{multline*}
\ _{2}D_{n_{1},n_{2}}\left(  x_{1},x_{2};a_{1},a_{2}\right)   =~_{3}%
F_{2}\left(
\genfrac{}{}{0pt}{0}{-n_{1},\ n_{1}+2\left(  n_{2}+a_{1}+a_{2}\right)
,\ a_{1}+\frac{x_{1}}{2}+\frac{n_{2}}{2}+\frac{1}{4}}{n_{2}+2a_{1}+\frac{1}%
{2},\ n_{2}+a_{1}+a_{2}+\frac{1}{2}}%
\mid1\right) \\
 \times~_{3}F_{2}\left(
\genfrac{}{}{0pt}{0}{-n_{2},\ n_{2}+2a_{1}+2a_{2}-1,\ a_{1}+\frac{x_{2}}%
{2}}{2a_{1},\ a_{1}+a_{2}}%
\mid1\right) \\
 \times \Gamma \left(  a_{1}+\frac{n_{2}+x_{1}}{2}+\frac{1}{4}\right)
\Gamma \left(  a_{1}+\frac{n_{2}-x_{1}}{2}+\frac{1}{4}\right)
  \Gamma \left(  a_{1}-\frac{x_{2}}{2}\right)  \Gamma \left(  a_{1}%
+\frac{x_{2}}{2}\right)  ,
\end{multline*}
which can be expressed in terms of the continuous Hahn polynomials \eqref{hahn} as
\begin{multline*}
\ _{2}D_{n_{1},n_{2}}\left(  x_{1},x_{2};a_{1},a_{2}\right)  =\frac
{n_{1}!n_{2}!i^{-n_{1}-n_{2}}}{\left(  2a_{1}\right)  _{n_{2}}\left(
a_{1}+a_{2}\right)  _{n_{2}}\left(  n_{2}+2a_{1}+\frac{1}{2}\right)  _{n_{1}%
}\left(  n_{2}+a_{1}+a_{2}+\frac{1}{2}\right)  _{n_{1}}}\\
 \times p_{n_{1}}\left(  \frac{-ix_{1}}{2};a_{1}+\frac{n_{2}}{2}+\frac{1}%
{4},a_{2}+\frac{n_{2}}{2}+\frac{1}{4},a_{1}+\frac{n_{2}}{2}+\frac{1}{4}%
,a_{2}+\frac{n_{2}}{2}+\frac{1}{4}\right) \\
 \times p_{n_{2}}\left(  \frac{-ix_{2}}{2};a_{1},a_{2},a_{2},a_{1}\right) \\
 \times \Gamma \left(  a_{1}+\frac{n_{2}+x_{1}}{2}+\frac{1}{4}\right)
\Gamma \left(  a_{1}+\frac{n_{2}-x_{1}}{2}+\frac{1}{4}\right)
 \Gamma \left(  a_{1}-\frac{x_{2}}{2}\right)  \Gamma \left(  a_{1}+\frac{x_{2}}{2}\right)  .
\end{multline*}
Similar to the cases $r=1$ and $r=2$, if we substitute (\ref{15}) and
(\ref{18}) in the Parseval identity \eqref{multi}, the necessary calculations give the desired result.
\end{proof}

\section*{Acknowledgements}
The research of the second author has been partially supported by TUBITAK Research Grant Proj. No. 120F140. The research of the third author has been partially supported by the Agencia Estatal de Investigaci\'on (AEI) of Spain under grant PID2020-113275GB-I00, cofinanced by the European Community fund FEDER.

\end{document}